\documentclass[reqno]{article}
\usepackage{pifont}
\usepackage{amsthm}
\usepackage{color}
\usepackage{amssymb}
\usepackage{mathrsfs}
\usepackage{amsmath, amsfonts, vmargin, enumerate}
\usepackage{graphics}
\usepackage{verbatim}
\usepackage{amsthm}
\usepackage{latexsym,bm}
\usepackage{euscript}
\usepackage[noblocks]{authblk}
\usepackage{geometry}
\usepackage{pgflibraryarrows}
\usepackage{pgflibrarysnakes}
\usepackage{subeqnarray}
\usepackage{cases}
\usepackage{xcolor}
\usepackage{subcaption}
\usepackage{caption}
\makeatletter

\renewcommand{\l}{\lambda}

\numberwithin{equation}{section}
\newtheorem{theorem}{Theorem}[section]

\newcommand\blfootnote[1]{%
  \begingroup
  \renewcommand\thefootnote{}\footnote{#1}%
  \addtocounter{footnote}{-1}%
  \endgroup
}

\newcommand{\Z}{\mathbb{Z}}

\renewcommand{\l}{\lambda}
\newcommand{\R}{\mathbb{R}}

\newcommand{\x}{\xi}
\newcommand{\am}{\operatorname{am}}
\newcommand{\cn}{\operatorname{cn}}
\newcommand{\sn}{\operatorname{sn}}
\newcommand{\dn}{\operatorname{dn}}

\newcommand{\cnap}[1]{\cn^{#1}(\l \x , m)}
\newcommand{\sna}{\sn(\l \x , m)}

\newcommand{\dna}{\dn(\l \x , m)}

\DeclareMathOperator{\sech}{sech}

\begin{document}

\title{Exact Jacobi elliptic solutions of the $abcd$-system}

\author{JAKE DANIELS$^1$ and NGHIEM V. NGUYEN$^1$ \\ 
}
\markboth{Jake Daniels and Nghiem V. Nguyen}
{Jacobian elliptic solutions}
\date{}
\maketitle

\begin{abstract}
    In this manuscript, consideration is given to the existence of periodic traveling-wave 
    solutions to the $abcd$-system.  This system was derived by Bona, Saut, and Chen to 
    describe small amplitude, long wavelength gravity waves on the surface of water. 
    These exact solutions are formulated in terms of the Jacobi elliptic function cnoidal. 
    The existence of explicit traveling-wave solutions is very useful in theoretical 
    investigations such as stability of solutions, as well as other numerical analysis 
    of the system.
\end{abstract}
\date{}
\maketitle
\blfootnote{{\it 2020 AMS Subject Classification:} 35A16, 35A24, 35B10.}
\blfootnote{{\it Key words:} Periodic solutions, Jacobi elliptic functions, 
    Cnoidal solutions, abcd system.}
\footnotetext[1]{Department of Mathematics and Statistics, 
    Utah State University, Logan, UT 84322-3900, USA}
\section{Introduction}
Let $\Omega_t$ be a domain in $\mathbb R^3\times \mathbb R$ whose boundary consists of two parts: the 
fixed surface located at $z=-h(x,y)$, and the free surface $z= \eta(x,y,t)$.  It is well-known 
that the three-dimensional capillary-gravity waves on an inviscid, irrotational, and 
incompressible fluid layer of uniform depth $h$ that is impermeable on the fixed surface are 
governed by the Euler equations
\begin{subequations}
\begin{numcases}{}
 \Delta \phi =0,\quad    \mbox{in}\  \Omega_t;\label{equ: poisson}\\
  \eta_t + \phi_x\eta_x + \phi_y\eta_y - \phi_z=0, \quad \text{at} \  z=\eta(x,y,t);\label{equ: boundary1}\\
  \phi_t+\frac{1}{2}|\nabla \phi|^2 + g z=0,\quad \text{at} \  z=\eta(x,y,t);\label{equ: boundary2}\\
  \phi_x h_x +\phi_y h_y +\phi_z=0,\quad \text{at} \ z=-h(x,y).\label{equ: boundary3}
\end{numcases}
\end{subequations}
Here, $g$ denotes the acceleration of gravity; $\phi(x,y,z,t)$ is the velocity potential 
function. The full Euler equations \eqref{equ: poisson}-\eqref{equ: boundary3} are often far 
more complex than  necessary for many applications, so several approximation
models have been derived in its place under certain restricted physical regimes.
One such model is the $abcd$-system, introduced by Bona \textit{et} \textit{al}.\ in \cite{BCS1,BCS2}, 
to describe small amplitude, long wavelength gravity waves on the surface of water 
\begin{equation}\label{abcd-system}
    \begin{cases}
        \eta_t + w_x + (w\eta)_x + aw_{xxx} - b\eta_{xxt} = 0 \,,\\
        w_t + \eta_x +ww_x + c\eta_{xxx} - dw_{xxt} = 0 \,,
    \end{cases}
\end{equation}
where $a, b, c,$ and $d$ are real constants and $\theta \in [0,1]$ that satisfy
\begin{equation}\label{abcd-parameters}
    \begin{split}
        a + b &= \frac{1}{2}(\theta^2 - \frac{1}{3}), \\
        c + d &= \frac{1}{2}(1 - \theta^2) \geq 0, \\
        a + b + c + d &= \frac{1}{3}.
    \end{split}
\end{equation}
Precisely, the regime for which the above system is derived to approximate the Euler equations 
is when the maximal deviation $\alpha$ of the free surface is small and a typical wavelength 
$\lambda$ is large as compared to the undisturbed water depth $h$, such that the Stokes number 
$S=\frac{\alpha \lambda^2}{h^3}$ is of order one. The functions $\eta(x,t)$ and $w(x,t)$ are 
real valued and $x, t \in \R$.  By choosing specific values for the parameters $a,b,c,$ and 
$d$, the system \eqref{abcd-system} includes a wide range of other systems that have been 
derived over the last few decades such as the classical Boussinesq system \cite{Bou1,Bou2,Bou3}, 
the Kaup system \cite{K}, the coupled Benjamin-Bona-Mahony system (BBM-system) \cite{BC}, 
the coupled Korteweg-de Vries system (KdV-system) \cite{BCS1, BCS2},
the Bona-Smith system \cite{BS}, and the integrable version of Boussinesq system \cite{Kri}. 
In particular, these specializations are:
\begin{itemize}
    \item Classical Boussinesq system
    \begin{equation*}
    \left\{
    \begin{matrix}
    \begin{split}
    & \eta_t +w_x +(w\eta)_x =0,\\
    & w_t +\eta_x +ww_x -\frac{1}{3}w_{xxt}=0;
    \end{split}
    \end{matrix}
    \right.
    \end{equation*}
    
    \item Kaup system
    \begin{equation*}
    \left\{
    \begin{matrix}
    \begin{split}
    & \eta_t +w_x +(w\eta)_x +\frac{1}{3} w_{xxx} =0,\\
    & w_t +\eta_x +ww_x =0;
    \end{split}
    \end{matrix}
    \right.
    \end{equation*}
    
    \item Bona-Smith system
    \begin{equation*}
    \left\{
    \begin{matrix}
    \begin{split}
    & \eta_t +w_x +(w\eta)_x -b\eta_{xxt} =0,\\
    & w_t +\eta_x +ww_x + c\eta_{xxx}- b w_{xxt}=0;
    \end{split}
    \end{matrix}
    \right.
    \end{equation*}
    
    \item Coupled BBM-system
    \begin{equation*}
    \left\{
    \begin{matrix}
    \begin{split}
    & \eta_t +w_x +(w\eta)_x -\frac{1}{6}\eta_{xxt} =0,\\
    & w_t +\eta_x +ww_x -\frac{1}{6}w_{xxt}=0;
    \end{split}
    \end{matrix}
    \right.
    \end{equation*}

    \item Coupled KdV-system
    \begin{equation*}
    \left\{
    \begin{matrix}
    \begin{split}
    & \eta_t +w_x +(w\eta)_x +\frac{1}{6} w_{xxx}=0,\\
    & w_t +\eta_x +ww_x +\frac{1}{6}\eta_{xxx}=0;
    \end{split}
    \end{matrix}
    \right.
    \end{equation*}
    
    \item Coupled KdV-BBM system
    \begin{equation*}
    \left\{
    \begin{matrix}
    \begin{split}
    & \eta_t +w_x +(w\eta)_x +\frac{1}{6}w_{xxx}=0,\\
    & w_t +\eta_x +ww_x -\frac{1}{6}w_{xxt}=0;
    \end{split}
    \end{matrix}
    \right.
    \end{equation*}
    
    \item Coupled BBM-KdV system
    \begin{equation*}
    \left\{
    \begin{matrix}
    \begin{split}
    & \eta_t +w_x +(w\eta)_x -\frac{1}{6}\eta_{xxt} =0,\\
    & w_t +\eta_x +ww_x +\frac{1}{6}\eta_{xxx}=0.
    \end{split}
    \end{matrix}
    \right.
    \end{equation*}
\end{itemize}
For a more detailed discussion on this, we refer our readers to the papers 
\cite{BCS1,BCS2}.

In this manuscript, attention is given to the existence of periodic traveling-wave 
solutions to the $abcd$-system.  A traveling-wave solution to the system 
\eqref{abcd-system} is a vector solution
$\big(\eta(x,t), w(x,t)\big)$ of the form 
\begin{equation}\label{traveling-wave}
    \eta(x,t) = \eta(x - \sigma t) \quad \text{and} \quad w(x,t) = 
    w(x - \sigma t),
\end{equation}
where $\sigma$ denotes the speed of the waves. Notice that when $\sigma>0$, one obtains 
a right-propagating solution, while for $\sigma<0$, one has a left-propagating solution.
This feature of bi-direction propagation of $(\eta, w)$ is a hallmark of the $abcd$-system 
in its derivation and is the main difference between system \eqref{abcd-system} and the one-way 
approximation models such as
\begin{equation*}
u_t + u_x + \frac{1}{p}(u^p)_x \pm L u_x =0, \ \ \ p>0,
\end{equation*}
(which is commonly referred to as the generalized KdV equation when $L = \partial_{xx}$, 
the generalized BBM equation when $L=-\partial_{xt}$, and simply the KdV and BBM equations, 
respectively, when $p=1$), and the cubic, nonlinear Schr\"odinger equation (NLS)
\begin{equation*}
i u_t + u_{xx} + |u|^2 u=0.
\end{equation*}
It would also be interesting to analyze the existence of periodic solutions when the two components 
propagate at different speeds, and/or when the two components travel in opposite directions.  
Such scenarios are not considered in the scope of this paper.

The most studied types of traveling-wave solutions are periodic traveling-wave solutions and solitary-wave solutions.  
Periodic traveling-wave solutions are self-explanatory, while solitary waves are smooth traveling-wave 
solutions that are symmetric around a \textit{single} maximum and rapidly decay to zero away from the 
maximum. Though less common, the term solitary wave is 
sometimes used to describe traveling-wave solutions that are symmetric around a single maximum but 
approach nonzero constants as $\xi = x - \sigma t \rightarrow \pm \infty$.
In this manuscript, we consider non-trivial solutions to be solutions where both components are non-constant, and 
semi-trivial solutions to be solutions where one component is a constant. 
The constant solution, $(f,g)$ for any $f,\,g \in \R$, is always a solution and will be referred to as the trivial solution.

The topic of existence of explicit periodic traveling-wave solutions for many single equations are readily 
available such as for the KdV equation \cite{DJ}, generalized KdV equation \cite{CP}, BBM equation \cite{AZ}, 
NLS equation \cite{Kam}, and Whitham equation \cite{EK}. In contrast, there are far less of these results for 
systems due to the complexity of the coupled equations involved. Leveling up the scalar results to systems 
is no minor task, requiring new insights and new approaches.  For the classical Boussinesq system, Krishnan 
showed that the periodic traveling-wave solutions exist and are given by Jacobi elliptic functions 
\cite{Kri}. Chen and Li \cite{CL} (see also \cite{WY}) have recently established the periodic traveling-wave 
solutions for the beta derivative of the Kaup system where the first temporal derivative $\partial_t$ is replaced 
by the fractional derivative $\partial^{\beta}_t$, with $\beta \in (0,1]$, given in terms of normal trigonometic 
functions as well as the Jacobi cnoidal functions. Four systems have been put forth recently  by Nguyen 
\textit{et al}.\ to describe the interaction of long and short waves in dispersive media \cite{DNS, LN, LN2}, and
periodic traveling-wave solutions given by the Jacobi elliptic function cnoidal have been explicitly calculated 
for all four systems \cite{BDN}.

Another important subject regarding traveling-wave solutions is the stability of these solutions.  The topic of stability of
periodic traveling-wave solutions such as cnoidal waves has attracted far less consideration 
than that of solitary wave solutions. Within this limited attention, explicitly known formulae for periodic traveling-wave 
solutions often play an instrumental role. Angula, Bona and Scialom \cite{ABS} proved that the periodic 
traveling-waves for the KdV equation are nonlinearly orbitally stable with respect to perturbations of the same period. 
Central to the argument in their proof is the fact that the cnoidal waves lie in the set of minimizers, a conclusion 
that would be hard to establish without having explicit formula for the cnoidal waves at hand. 
Deconinck \textit{et} \textit{al}.\ established that the KdV cnoidal waves are spectrally stable with respect 
to perturbations of arbitrary period \cite{BD}, and orbitally stable with respect to sub-harmonic perturbations \cite{DK}.  
Again, the explicit formulae for cnoidal waves play pivital roles in their numerical computations.

In this paper, periodic traveling-wave solutions to the system \eqref{abcd-system} of the form
\begin{equation}\label{eqn:eta-w}
    \eta(\x) = \sum_{r = 0}^n j_r \cnap{r} \quad \text{and} \quad 
    w(\x) = \sum_{r = 0}^n k_r \cnap{r}\,,
\end{equation}
are analyzed, where $j_r,\,k_r \in \mathbb R$, $\lambda > 0$, $ m \in (0,1]$, $\xi = x - \sigma t$, and $\cn$ denotes the Jacobi 
elliptic function cnoidal. For the readers' convenience, a brief introduction to Jacobi elliptic functions can be found in Section 2. 
If $m=0$, then \eqref{eqn:eta-w} reduces to a cosine series which
is not the aim of this paper and is excluded. The analysis of \eqref{eqn:eta-w} requires a case split for when $c \ne 0$ and $c=0$. 
For the first case, $c \ne 0$, it is shown that non-trivial periodic traveling-wave solutions only exist when $a^2 + b^2 \ne 0$, 
and in particular, $j_r = k_r = 0$ for $r \geq 3$.
When $a^2 + b^2 = 0$, it is established that only semi-trivial periodic traveling-wave solutions exist, and that
$j_r = 0$ for all $r \geq 1$ and $k_r = 0$ for $r \geq 3$.
For the second case, $c = 0$, it is shown that non-trivial periodic traveling-wave solutions only exist when $b^2 + d^2 \ne 0$,
and that $j_r = 0$ for $r = 1,\,3$ and $r \geq 5$ and $k_r = 0$ for $r = 1$ 
and $r \geq 3$.

The paper is organized as follows. In Section 2, some facts and identities of Jacobi elliptic functions are 
reviewed. In Section 3, the existence of periodic traveling-wave solutions to the system \eqref{abcd-system}
is established. The periodic traveling-wave solutions are then explicitly computed in Section 4.  
Finally, a discussion about the obtained results is given in Section 5, including remarks about how these 
cnoidal solutions limit to the solitary-wave solutions previously established in \cite{C1, C2}.

\section{Preliminaries}
For the readers' convinience, some notions of the Jacobi elliptic 
functions are briefly recalled here. Let
\begin{equation*}
    v=\int_0^{\phi} \frac{1}{\sqrt{1-m^2 \sin^2 t}}dt, \ \ \ \ \ 
    \mbox{for $0\leq m\leq 1$}.
\end{equation*}
Then $v =F(\phi, m)$, or equivalently, $\phi = F^{-1}(v,m)= \am(v,m)$ 
which is the Jacobi amplitude and $m$ is referred to as the Jacobi elliptic modulus. The two basic Jacobi elliptic 
functions $\cn(v,m)$ and $\sn(v,m)$ are defined as
\begin{equation*}
    \sn(v,m) = \sin(\phi) = \sin\big(F^{-1}(v,m)\big) \ \ \ \ \ 
    \mbox{and} \ \ \ \ \ \cn(v,m)= \cos(\phi) = \cos\big(F^{-1}(v,
    m)\big).
\end{equation*}
 These 
functions are generalizations of the trigonometric and hyperbolic 
functions which satisfy
\begin{equation*}
    \begin{matrix}
        \cn(v,0) =\cos(v), \ \ \ \ \ \ \ \sn(v,0) =\sin(v),\\
        \cn(v,1) = \operatorname{sech} (v), \ \ \ \ \ \sn(v,1) =\tanh 
            (v).
    \end{matrix}
\end{equation*}
We recall the following identities:
\begin{equation*}
    \left\{
    \begin{matrix}
    \begin{split}
    &\sn^2(\lambda \xi,m) =1-\cn^2 (\lambda \xi,m),\\
    &\dn^2(\lambda \xi,m) =1 - m^2 + m^2 \cn^2(\lambda \xi,m),\\
    &\frac{d}{d\xi}\cn(\lambda \xi,m)  = -\lambda \sn(\lambda \xi,m) \dn(\lambda \xi,m),\\
    &\frac{d}{d\xi} \sn(\lambda \xi,m) = \lambda \cn(\lambda \xi,m) \dn(\lambda \xi,m),\\
    &\frac{d}{d\xi} \dn(\lambda \xi,m) = -m^2 \lambda \cn(\lambda \xi,m) \sn(\lambda \xi,m).
    \end{split}
    \end{matrix}
    \right.
\end{equation*}
The following relations will be useful
\begin{equation}\label{derivatives}
    \left\{
    \begin{matrix}
    \begin{split}
    &\frac{d}{d\xi}\cn^r  = -r\lambda \cn^{r-1} \sn \dn,\\
    &\frac{d^2}{d\xi^2} \cn^r = -r \lambda^2\big[(r+1) m^2 \cn^{r+2} + r (1-2m^2) \cn^r + (r-1) (m^2-1) \cn^{r-2}\big],\\
    &\frac{d^3}{d\xi^3} \cn^r = r \lambda^3 \sn \dn \big[(r+1) (r+2) m^2 \cn^{r+1} + r^2 (1-2m^2) \cn^{r-1} + (r-1) (r-2) (m^2-1) \cn^{r-3}\big], \\
    \end{split}
    \end{matrix}
    \right.
    \end{equation}
where the argument $(\lambda \xi,m)$ has been dropped for clarity 
reason.

\section{Existence of Periodic Traveling-Wave Solutions}
Substituting the traveling-wave ansatz 
\eqref{traveling-wave} into the abcd-system \eqref{abcd-system} 
yields the following system of ordinary differential equations (ODE)
\begin{equation}\label{eta-w-system}\begin{cases}
    -\sigma \eta' + w' + (\eta w)' + aw''' + b \sigma \eta''' = 0 \,,\\
    -\sigma w' + \eta' + ww' + c\eta''' + d \sigma w''' = 0 \,,
\end{cases}\end{equation}
where the primes denote the derivatives with respect to the moving 
frame $\xi = x - \sigma t$. Replacing $\eta$ and $w$ in \eqref{eta-w-system} with the 
periodic traveling-wave ansatz \eqref{eqn:eta-w} and using cnoidal derivative identities
\eqref{derivatives}, the following generic form is obtained
\begin{equation}\label{eqn:h-pq}
    \left\{
    \begin{matrix}
    \begin{split}
        &- \lambda \sna \dna \sum_{q = 0}^{2n-1} h_{1, q} \cnap{q} = 0, \\
        &- \lambda \sna \dna \sum_{q = 0}^{2n-1} h_{2, q} \cnap{q} = 0,
    \end{split}
    \end{matrix}
    \right.
\end{equation}
where the subscript $q$ in $h_{1,q}$ and $h_{2,q}$ corresponds with the 
power of $\cn$ in their respective equations.  

The analysis is now split into two distinct cases. We will show that the 
the series \eqref{eqn:eta-w} will terminate after 
a certain number of terms. When exactly this termination will occur is
found to be dependent on the values of the parameters $a,\, b,\, c,$ and $d$.

\subsection{The case when $c \ne 0$}
Notice that as 
\eqref{eqn:h-pq} must hold true for all $(\lambda \xi,m)$, it must be 
the case that $h_{1,q}=0$ and $h_{2,q}=0$ for all $q$.
Using this, the following theorem is established.
(Recall that $j_r$ and $k_r$ are the coefficients of $\cn^r$ as given 
in \eqref{eqn:eta-w} for $r \in \Z_{\geq 0}$.)

\begin{theorem}\label{theorem1}
    Suppose $c \ne 0$, then the periodic traveling-wave ansatz \eqref{eqn:eta-w}
    will take on one of the following forms depending on the values of $a$ and $b$:
    \begin{enumerate}[(i)]
        \item If $a^2 + b^2 \ne 0$, then $j_r = k_r = 0$
            for all integers $r \geq 3$ and \eqref{eqn:eta-w} reduces to
            \begin{equation}\label{eqn:eta-w-1}
                \eta(\x) = j_0 + j_1 \cn(\l \x, m) + j_2 \cn^2 (\l \x, m) 
                \quad \text{and} \quad 
                w(\x) = k_0 + k_1 \cn(\l \x, m) + k_2 \cn^2 (\l \x, m)\,.
            \end{equation}
        \item If $a^2 + b^2 = 0$, then $j_r = 0$ for all integers 
            $r \geq 1$ and $k_r = 0$ for all integers $r \geq 3$. Thus, 
            \eqref{eqn:eta-w} reduces to
            \begin{equation}\label{eqn:eta-w-ST}
                \eta(\x) = j_0  \quad \text{and} \quad 
                w(\x) = k_0 + k_1 \cn (\l \x, m) + k_2 \cn^2 (\l \x, m)\,.
            \end{equation}
            Consequently, only semi-trivial periodic traveling-wave solutions exist 
            in this case.
    \end{enumerate}
\end{theorem}
\begin{proof}
    \textit{(i)} Let $\rho(\eta')$ denote the largest power of $\cn$ 
    in $\eta'$. Then,
    \begin{center}
        $ \rho(\eta')=\rho(w')=n-1 $, \quad $ \rho(\eta''')=\rho(w''')=n+1 $, \quad and \quad  
        $ \rho((\eta w)') = \rho(ww')=2n-1 $.
    \end{center}
    From this, if $n>2$ then $2n -1> n+1$, which implies that the coefficient of
    $\cn^{2n-1}$ in the second equation of \eqref{eqn:h-pq} will solely
    come from the $ww'$-term. Thus, $h_{2, 2n-1} =
    n k_n^2$. Setting this coefficient equal to zero 
    gives $k_{n}=0$. Recalculating the above orders yields
    \begin{center}
        $ \rho(\eta')= n-1$, \quad $\rho(w')=n-2$, \quad $ \rho(\eta''')=n+1$, \quad 
        $\rho(w''')=n $, \quad  
        $ \rho((\eta w)') = 2n-2$, \quad and \quad $\rho(ww')=2n-3 $.
    \end{center}
    After $i$ iterations, the orders become
    \begin{center}
        $ \rho(\eta')= n-1$, \quad $\rho(w')=n-2-i$, \quad $ \rho(\eta''')=n+1$, \quad 
        $\rho(w''')=n-i $, \quad  
        $ \rho((\eta w)') = 2n-2-i$, \quad and \quad $\rho(ww')=2n-3-2i $.
    \end{center}
    This argument can be repeated until $\rho(ww') \leq \rho(\eta''')$.  
    Now, consider the two possibilities for $n$.
    \begin{itemize}
        \item If n is even: let $z \in {\Z}_{\geq 1}$ such that 
        $n=2z$. Then $\rho(ww')=2n-3-2i=4z-3-2i$ and $\rho(\eta''')=n+1=2z+1$. 
        Since $\rho(ww')$ and $\rho(\eta''')$ are both odd, there exists an 
        integer $i$ such that $\rho(ww') = \rho(\eta''')$. Solving for $i$ gives
        \begin{align*}
            4z-3-2i &= 2z+1.
        \end{align*}
        Thus, $\rho(ww')=\rho(\eta''')$ when $i=z-2$. Recalculating the 
        orders reveals
        \begin{center}
            $ \rho(\eta')= 2z-1$, \quad $\rho(w')=z$, \quad $ \rho(\eta''')=2z+1$, \quad 
            $\rho(w''')=z+2 $, \quad  
            $ \rho((\eta w)') = 3z$, \quad and \quad $\rho(ww')=2z+1 $.
        \end{center}
        \item If n is odd: let $z \in {\Z}_{\geq 1}$ such that 
        $n=2z+1$. Then $\rho(ww')=2n-3-2i=4z-1-2i$ and $\rho(\eta''')=n+1=2z+2$. 
        Since $\rho(ww')$ is odd and $\rho(\eta''')$ is even, there 
        will not be an integer $i$ such that $\rho(ww') = \rho(\eta''')$. So this
        argument will break when $\rho(ww') < \rho(\eta''')$, which will happen when
        \begin{align*}
            4z-1-2i &< 2z+2,
        \end{align*}
        which gives $i > z - \frac{3}{2}$. Thus, $i = z-1$ will be the first iteration 
        such that the above inequality is achieved. From here, the orders become
        \begin{center}
            $ \rho(\eta')= 2z$, \quad $\rho(w')=z$, \quad $ \rho(\eta''')=2z+2$, \quad 
            $\rho(w''')=z+2 $, \quad  
            $ \rho((\eta w)') = 3z+1$, \quad and \quad $\rho(ww')=2z+1 $.
        \end{center}
        However, $\rho(\eta''')>\rho(ww')$, so from the second 
        equation, the coefficient of $\cn^{2z+2}$ will solely come from 
        the $\eta'''$-term, which implies $h_{2, 2z+2} = -c \l^2 m^2 (2z+1)(2z+2)(2z+3)
        j_{2z+1}$. Since $c, \l, m$ are not zero, it must be the case that $j_{2z+1}=0$.
        Recalculating the orders reveals
        \begin{center}
            $ \rho(\eta')= 2z-1$, \quad $\rho(w')=z$, \quad $ \rho(\eta''')=2z+1$, \quad 
            $\rho(w''')=z+2 $, \quad  
            $ \rho((\eta w)') = 3z$, \quad and \quad $\rho(ww')=2z+1 $.
        \end{center}
    \end{itemize}
    So, regardless of whether $n$ is even or odd, the same reduction of orders is achieved.
    The orders of the second equation are now balanced, but the orders in the first 
    equation are not. Specifically, $\rho((\eta w)')$ is greater than any 
    other term in the first equation. This means that the 
    coefficient of $\cn^{3z}$ in the first equation of \eqref{eqn:h-pq} 
    will solely come from the $(\eta w)'$-term, which implies $h_{1, 3z} = (3z+1)
    j_{2z}k_{z+1}$. Demanding $h_{1, 3z}=0$ yields that either $j_{2z} = 0$
    or $k_{z+1} = 0$.
    \begin{itemize}
        \item If $k_{z+1} = 0$: recalculating the orders
        \begin{center}
            $ \rho(\eta')= 2z-1$, \quad $\rho(w')=z-1$, \quad $ \rho(\eta''')=2z+1$, \quad 
            $\rho(w''')=z+1 $, \quad  
            $ \rho((\eta w)') = 3z-1$, \quad and \quad $\rho(ww')=2z-1 $.
        \end{center}
        Like in the odd case above, $\rho(\eta''')>\rho(ww')$, so the coefficients of 
        $\cn^{2z+1}$ and $\cn^{2z}$ are solely given by the $\eta'''$-term. 
        This implies that $h_{2, 2z+1} = -c\l^2 m^2 (2z)(2z+1)(2z+2)j_{2z}$.
        Since $c, \l, m$ are not zero, it must be the case that $j_{2z} = 0$.
        Similarly, since $h_{2, 2z} = -c\l^2m^2(2z-1)(2z)(2z+1)j_{2z-1}$, the same 
        logic yields $j_{2z-1} = 0$. Thus,
        \begin{center}
            $ \rho(\eta')= 2z-3$, \quad $\rho(w')=z-1$, \quad $ \rho(\eta''')=2z-1$, \quad 
            $\rho(w''')=z+1 $, \quad  
            $ \rho((\eta w)') = 3z-3$, \quad and \quad $\rho(ww')=2z-1 $.
        \end{center} 
        \item If $j_{2z}=0$: recalculating the orders
        \begin{center}
            $ \rho(\eta')= 2z-2$, \quad $\rho(w')=z$, \quad $ \rho(\eta''')=2z$, \quad 
            $\rho(w''')=z+2 $, \quad  
            $ \rho((\eta w)') = 3z-1$, \quad and \quad $\rho(ww')=2z+1 $.
        \end{center}
        Now, $\rho(ww')>\rho(\eta''')$, so it is similarly deduced that  
        $h_{2, 2z+1} = (z+1) k_{z+1}^2$ which implies $k_{z+1} = 0$. 
        Therefore,
        \begin{center}
            $ \rho(\eta')= 2z-2$, \quad $\rho(w')=z-1$, \quad $ \rho(\eta''')=2z$, \quad 
            $\rho(w''')=z+1 $, \quad  
            $ \rho((\eta w)') = 3z-2$, \quad and \quad $\rho(ww')=2z-1 $.
        \end{center}
        Again $\rho(\eta''')>\rho(ww'),$ so $j_{2z-1} = 0$. 
        Consequently,
        \begin{center}
            $ \rho(\eta')= 2z-3$, \quad $\rho(w')=z-1$, \quad $ \rho(\eta''')=2z-1$, \quad 
            $\rho(w''')=z+1 $, \quad  
            $ \rho((\eta w)') = 3z-3$, \quad and \quad $\rho(ww')=2z-1 $.
        \end{center}
    \end{itemize}
    So, regardless of whether 
    $j_{2z} = 0$ or $k_{z+1} = 0$, the same order reduction is achieved. 
    After $i$ iterations of this argument, the orders become
    \begin{center}
        $ \rho(\eta')= 2z-3-2i$, \quad $\rho(w')=z-1-i$, \quad 
        $ \rho(\eta''')=2z-1-2i$, \quad $\rho(w''')=z+1-i $, \quad  
        $ \rho((\eta w)') = 3z-3-3i$, \quad and \quad $\rho(ww')=2z-1-2i $.
    \end{center}
    By balancing the orders, there are three instances 
    in which this argument breaks. From the first equation, the argument
    breaks when $\rho((\eta w)') \leq \rho(w''')$ or $\rho((\eta w)') \leq \rho(\eta''')$. 
    From the second equation, the argument 
    will break when $\rho(ww') = \rho(\eta''') \leq \rho(w''')$.
    \begin{center}
        $\rho((\eta w)') \leq \rho(w'''): \quad 
            3z-3-3i \leq z+1-i \quad \iff \quad z-2 \leq i$; \\
        $\rho((\eta w)') \leq \rho(\eta'''):
            \quad 3z-3-3i \leq 2z-1-2i \quad \iff \quad  z-2 \leq i$; \\
        $\rho(ww') = \rho(\eta''') \leq \rho(w'''):
            \quad 2z-1-2i \leq z+1-i \quad \iff \quad z-2 \leq i$. \\
    \end{center}
    So this argument will break after $(z-2)$ iterations for 
    all three cases.
    Plugging $i=z-2$ and recalculating the 
    orders a final time yields
    \begin{equation}\label{eqn:orders-1}
        \rho(\eta')= 1, \quad \rho(w')=1, \quad 
        \rho(\eta''')=3, \quad \rho(w''')=3 , \quad  
        \rho((\eta w)') = 3, \quad \text{and} \quad \rho(ww')=3.
    \end{equation}
    Since $c \ne 0$ and $a^2 + b^2 \ne 0$, the 
    highest order terms in both equations are balanced. Furthermore,
    since $\rho(\eta') = \rho(w') = 1$, the identities in \eqref{derivatives}
    imply that $\rho(\eta) = \rho(w) = 2$.
    Therefore, if $c \ne 0$ and $a^2 + b^2 \ne 0$,
    then $j_r = k_r = 0$, for all $r \geq 3$. \\[10 pt]
    \textit{(ii)}
    If $a^2 + b^2 =0$, then $a = b = 0$. With this, the
    first equation in \eqref{eta-w-system} is reduced to
    \begin{equation*}
        -\sigma \eta' + w' + (\eta w)' = 0.
    \end{equation*}
    The orders in \eqref{eqn:orders-1} are now unbalanced since $\rho((\eta w)') > 
    \rho(\eta') = \rho(w')$. So, $h_{1, 3} = 4j_2k_2$ which implies either 
    $j_2 = 0$ or $k_2 = 0$.
    \begin{itemize}
        \item If $j_2 = 0$: then it follows that $h_{1,2} = 3j_1k_2$, so either 
            $j_1 = 0$ or $k_2 = 0$. 
        \begin{itemize}
            \item If $j_1 = 0$, then the desired result is obtained. 
            \item If $k_2 = 0$, then the orders in the second equation of \eqref{eta-w-system} 
            are unbalanced, and $h_{2, 1} = k_1^2$ which yields $k_1 = 0$. Now $\rho(w w') < 
            \rho(\eta''')$ so from the same argument used in the proof of \textit{(i)}, it follows that
            $j_1 = 0$.
        \end{itemize}
        \item If $k_2 = 0$: then $h_{1,2} = 3j_2k_1$, so either $j_2 = 0$ or $k_1 = 0$.
        \begin{itemize}
            \item If $j_2 = 0$, then we arrive at the same case as above, and it 
                follows that $j_1 = k_1 = 0$.
            \item If $k_1 = 0$, then $w = k_0$ which implies $w w' = 0$. Now, the 
                second equation in \eqref{eta-w-system} becomes 
                \begin{equation*}
                    \eta' + c \eta''' = 0,
                \end{equation*}
                but $\rho(\eta''') > \rho(\eta')$ and it quickly follows that 
                $j_2 = j_1 = 0$.
        \end{itemize}
    \end{itemize}

    Therefore, when $c \ne 0$ and $a^2 + b^2 = 0$, $j_r = 0$ for all integers $r \geq 1$.
    Moreover, the same result proved in \textit{(i)} holds for $w$, that is, 
    $k_r = 0$ for all integers $r \geq 3$. Thus, $\eta$ is constant so there are only 
    semi-trivial periodic traveling-wave solutions in this case. \\
\end{proof}

From the above analysis, we see that non-trivial periodic traveling-wave solutions to 
\eqref{eta-w-system} may only occur when $c \ne 0$ and $a^2 + b^2 \ne 0$, so 
we will focus on this case. 
Under these conditions, $j_r = k_r = 0$ for
all integers $r \geq 3$, so \eqref{eqn:h-pq} reduces to
\begin{equation*}
    \left\{
    \begin{matrix}
    \begin{split}
        &- \lambda \sna \dna \sum_{q = 0}^{3} h_{1, q} \cnap{q} = 0, \\
        &- \lambda \sna \dna \sum_{q = 0}^{3} h_{2, q} \cnap{q} = 0.
    \end{split}
    \end{matrix}
    \right.
\end{equation*}

Next, by demanding $h_{1, q}=0$ and $h_{2, q}=0$, 
we obtain the following system of 8 equations with 9 unknowns, $\lambda,\, m,\, 
\sigma,\, j_i,$ and $k_i$ for $i=0,1,2$:
\begin{equation}\label{coeffs 1}\begin{cases}
    h_{1,3} = -24\,b{\lambda}^{2}{m}^{2}\sigma\,j_{{2}}-24\,a{\lambda}^{2}{m}^{2}k_{
        {2}}+4\,j_{{2}}k_{{2}}, \\[5 pt]

    h_{1,2} = -6\,b{\lambda}^{2}{m}^{2}\sigma\,j_{{1}}-6\,a{\lambda}^{2}{m}^{2}k_{{1
    }}+3\,j_{{1}}k_{{2}}+3\,j_{{2}}k_{{1}}
    , \\[5 pt]

    h_{1,1} = 16\,b{\lambda}^{2}{m}^{2}\sigma\,j_{{2}}+16\,a{\lambda}^{2}{m}^{2}k_{{
        2}}-8\,b{\lambda}^{2}\sigma\,j_{{2}}-8\,a{\lambda}^{2}k_{{2}}-2\,
        \sigma\,j_{{2}}+2\,j_{{0}}k_{{2}} \\
        \qquad +2\,j_{{1}}k_{{1}}+2\,j_{{2}}k_{{0}}+
        2\,k_{{2}}
        , \\[5 pt]

    h_{1,0} = 2\,b{\lambda}^{2}{m}^{2}\sigma\,j_{{1}}+2\,a{\lambda}^{2}{m}^{2}k_{{1}
    }-b{\lambda}^{2}\sigma\,j_{{1}}-a{\lambda}^{2}k_{{1}}-\sigma\,j_{{1}}+
    j_{{0}}k_{{1}}+j_{{1}}k_{{0}}+k_{{1}}
    , \\[5 pt]

    h_{2,3} = -24\,d{\lambda}^{2}{m}^{2}\sigma\,k_{{2}}-24\,c{\lambda}^{2}{m}^{2}j_{
        {2}}+2\,k_{{2}}^{2}
        , \\[5 pt]

    h_{2,2} = -6\,d{\lambda}^{2}{m}^{2}\sigma\,k_{{1}}-6\,c{\lambda}^{2}{m}^{2}j_{{1
    }}+3\,k_{{1}}k_{{2}}
    , \\[5 pt]

    h_{2,1} = 16\,d{\lambda}^{2}{m}^{2}\sigma\,k_{{2}}+16\,c{\lambda}^{2}{m}^{2}j_{{
        2}}-8\,d{\lambda}^{2}\sigma\,k_{{2}}-8\,c{\lambda}^{2}j_{{2}}-2\,
        \sigma\,k_{{2}}+2\,k_{{0}}k_{{2}}\\
        \qquad +k_{{1}}^{2}+2\,j_{{2}}
        , \\[5 pt]

    h_{2,0} = 2\,d{\lambda}^{2}{m}^{2}\sigma\,k_{{1}}+2\,c{\lambda}^{2}{m}^{2}j_{{1}
    }-d{\lambda}^{2}\sigma\,k_{{1}}-c{\lambda}^{2}j_{{1}}-\sigma\,k_{{1}}+
    k_{{0}}k_{{1}}+j_{{1}}\,.
\end{cases}\end{equation}

The exact periodic traveling-wave solutions to \eqref{eta-w-system}
could then be established by solving the system of nonlinear 
equations \eqref{coeffs 1} with the help of the computer software Maple.
As there is one degree of freedom, any of the 9 unknowns
can be chosen as a free parameter. In most physical settings, it is 
desirable to have the elliptic modulus $m$ as the free parameter. 
With this in mind, the solutions found in the next section
will be given in terms of $m$, along with the systems' parameters
$a,\, b,\, c,$ and $d$. 
In Section 4.1, a solution is found with $j_1=k_1=0$.
For this special case, \eqref{coeffs 1} is reduced 
to a system of 4 equations with 7 unknowns which creates two new
degrees of freedom. In this case $\lambda$ and $\sigma$, the wavelength 
and the speed of the wave respectively, will be chosen along with $m$ 
as the free parameters in our solutions.  Furthermore, 
semi-trivial solutions for the case when $c \ne 0$
and $a^2 + b^2 = 0$ will be presented in Section 4.3.

\subsection{The case when $c = 0$}
When $c=0$, system \eqref{eta-w-system} reduces to
\begin{equation}\label{eta-w-system-c0}\begin{cases}
    -\sigma \eta' + w' + (\eta w)' + aw''' + b \sigma \eta''' = 0 \,,\\
    -\sigma w' + \eta' + ww' + d \sigma w''' = 0 \,.
\end{cases}\end{equation}
In addition, \eqref{eqn:h-pq} will take on the form
\begin{equation}\label{eqn:h-pq-c0}
    \left\{
    \begin{matrix}
    \begin{split}
        &- \lambda \sna \dna \sum_{q = 0}^{2n-1} \tilde{h}_{1, q} \cnap{q} = 0, \\
        &- \lambda \sna \dna \sum_{q = 0}^{2n-1} \tilde{h}_{2, q} \cnap{q} = 0,
    \end{split}
    \end{matrix}
    \right.
\end{equation}
where $\tilde{h}_{1,q} = h_{1,q} \rvert_{c=0}$ and $\tilde{h}_{2,q} = h_{2,q} 
\rvert_{c=0}$. 

Notice that as \eqref{eqn:h-pq-c0} must hold true for all $(\lambda \xi,m)$,
it must be the case that $\tilde{h}_{1,q}=0$ and $\tilde{h}_{2,q}=0$ for each $q$. 
Using this, the following theorem is established. (Recall that $j_r$ and $k_r$ are 
the coefficients of $\cn^r$ as given in \eqref{eqn:eta-w} for $r \in \Z_{\geq 0}$.)

\begin{theorem}\label{theorem2}
    Suppose $c = 0$, then the periodic traveling-wave ansatz \eqref{eqn:eta-w}
    will take on one of the following forms depending on the values of $b$ and $d$:
    \begin{enumerate}[(i)]
        \item If $b^2 + d^2 \ne 0$, then $j_r=0$ for all integers $r \geq 5$, and $k_r = 0$
            for all integers $r \geq 3$. Thus, \eqref{eqn:eta-w} reduces to
            \begin{equation*}
                \begin{split} 
                \eta(\x) &= j_0 + j_1 \cn(\l \x, m) + j_2 \cn^2 (\l \x, m)
                + j_3 \cn^3 (\l \x, m) + j_4 \cn^4 (\l \x, m) \\ 
                & \qquad \text{and} \quad
                w(\x) = k_0 + k_1 \cn(\l \x, m) + k_2 \cn^2 (\l \x, m)\,.
                \end{split}
            \end{equation*}
        \item If $b^2 + d^2 = 0$, then $j_r = k_r = 0$ for all integers 
            $r \geq 1$.  Consequently, only the trivial solution exists 
            in this case.
    \end{enumerate}
\end{theorem}
\begin{proof}
    \textit{(i)} Let $\rho(\eta')$ denote the largest power of $\cn$ 
    in $\eta'$. Then,
    \begin{center}
        $ \rho(\eta')=\rho(w')=n-1 $, \quad $ \rho(\eta''')=\rho(w''')=n+1 $, \quad and \quad  
        $ \rho((\eta w)') = \rho(ww')=2n-1 $.
    \end{center}
    If $n>2$ then $2n -1> n+1$, which implies the coefficient of
    $\cn^{2n-1}$ in the second equation of \eqref{eqn:h-pq-c0} will solely
    come from the $ww'$-term. Thus, $\tilde{h}_{2, 2n-1} =
    n k_n^2$. Setting this coefficient equal to zero 
    gives $k_{n}=0$. Recalculating the above orders yields
    \begin{center}
        $ \rho(\eta')= n-1$, \quad $\rho(w')=n-2$, \quad $ \rho(\eta''')=n+1$, \quad 
        $\rho(w''')=n $, \quad  
        $ \rho((\eta w)') = 2n-2$, \quad and \quad $\rho(ww')=2n-3 $.
    \end{center}
    After $i$ iterations of this argument, the orders 
    become
    \begin{center}
        $ \rho(\eta')= n-1$, \quad $\rho(w')=n-2-i$, \quad 
        $ \rho(\eta''')=n+1$, \quad $\rho(w''')=n-i $, \quad  
        $ \rho((\eta w)') = 2n-2-i$, \quad and \quad $\rho(ww')=2n-3-2i $.
    \end{center}
    In the proof of Theorem \ref{theorem1}, it was shown that this argument
    can be repeated until $\rho(ww') \leq \rho(\eta''')$, but for the $c=0$
    case, the $\eta'''$-term is no longer present in the second equation of 
    \eqref{eta-w-system}. Thus, the argument will break when $\rho\big((\eta w)'\big) \leq 
    \rho(\eta')$.  Now, consider the two possibilities for $n$.
    \begin{itemize}
        \item If n is even: let $z \in {\Z}_{\geq 1}$ such that 
        $n=2z$. Then $\rho(ww')=2n-3-2i=4z-3-2i$ and $\rho(\eta')=n-1=2z-1$. 
        Since $\rho(ww')$ and $\rho(\eta')$ are both odd, there exists an 
        integer $i$ such that $\rho(ww') = \rho(\eta')$. Solving for $i$ gives
        \begin{align*}
            4z-3-2i &= 2z-1.
        \end{align*}
        Thus, $\rho(ww')=\rho(\eta')$ when $i=z-1$. Recalculating the 
        orders yields
        \begin{center}
            $ \rho(\eta')= 2z-1$, \quad $\rho(w')=z-1$, \quad 
            $ \rho(\eta''')=2z+1$, \quad $\rho(w''')=z+1 $, \quad  
            $ \rho((\eta w)') = 3z-1$, \quad and \quad $\rho(ww')=2z-1 $.
        \end{center}
        \item If n is odd: let $z \in {\Z}_{\geq 1}$ such that 
        $n=2z+1$. Then $\rho(ww')=2n-3-2i=4z-1-2i$ and $\rho(\eta')=n-1=2z$. 
        Since $\rho(ww')$ is odd and $\rho(\eta')$ is even, there 
        will not be an integer $i$ such that $\rho(ww') = \rho(\eta')$. So this
        argument will break when $\rho(ww') < \rho(\eta')$, which will happen when
        \begin{align*}
            4z-1-2i &< 2z,
        \end{align*}
        which yields $i > z - \frac{1}{2}$. Thus, $i = z$ will be the first iteration 
        such that the above inequality is achieved. From here, the orders become
        \begin{center}
            $ \rho(\eta')= 2z$, \quad $\rho(w')=z-1$, \quad 
            $ \rho(\eta''')=2z+2$, \quad $\rho(w''')=z+1 $, \quad  
            $ \rho((\eta w)') = 3z$, \quad and \quad $\rho(ww')=2z-1 $.
        \end{center}
        However, $\rho(\eta')>\rho(ww')$, so the coefficient of $\cn^{2z}$ 
        in the second equation will solely come from 
        the $\eta'$-term, which implies $\tilde{h}_{2, 2z} = (2z+1) j_{2z+1}$. 
        It then follows that $j_{2z+1}=0$. Recalculating the orders yields
        \begin{center}
            $ \rho(\eta')= 2z-1$, \quad $\rho(w')=z-1$, \quad 
            $ \rho(\eta''')=2z+1$, \quad $\rho(w''')=z+1 $, \quad  
            $ \rho((\eta w)') = 3z-1$, \quad and \quad $\rho(ww')=2z-1 $.
        \end{center}
    \end{itemize}
    So, regardless of whether $n$ is even or odd, the same reduction of orders is achieved.
    Now the orders of the second equation are balanced, but the orders in the first 
    equation are not. Specifically, $\rho((\eta w)')$ is greater than any 
    other term in the first equation. This means that the 
    coefficient of $\cn^{3z-1}$ in the first equation of \eqref{eqn:h-pq} 
    will solely come from the $(\eta w)'$-term, which implies $\tilde{h}_{1, 3z-1} = 3z
    j_{2z}k_{z}$. Demanding $\tilde{h}_{1, 3z-1}=0$ reveals that either $j_{2z} = 0$
    or $k_{z} = 0$.
    \begin{itemize}
        \item If $k_{z} = 0$: recalculating the orders
        \begin{center}
            $ \rho(\eta')= 2z-1$, \quad $\rho(w')=z-2$, \quad 
            $ \rho(\eta''')=2z+1$, \quad $\rho(w''')=z $, \quad  
            $ \rho((\eta w)') = 3z-2$, \quad and \quad $\rho(ww')=2z-3 $.
        \end{center}
        Like in the odd case above, $\rho(\eta')>\rho(ww')$, so the coefficients of 
        $\cn^{2z-1}$ and $\cn^{2z-2}$ are solely given by the $\eta'$-term. 
        This implies that $\tilde{h}_{2, 2z-1} = 2z j_{2z}$, so it follows that $j_{2z} = 0$.
        Similarly since $\tilde{h}_{2, 2z-2} = (2z-1) j_{2z-1}$, the same 
        logic yields $j_{2z-1} = 0$.  Thus,
        \begin{center}
            $ \rho(\eta')= 2z-3$, \quad $\rho(w')=z-2$, \quad 
            $ \rho(\eta''')=2z-1$, \quad $\rho(w''')=z $, \quad  
            $ \rho((\eta w)') = 3z-4$, \quad and \quad $\rho(ww')=2z-3 $.
        \end{center}
        \item If $j_{2z}=0$: recalculating the orders
        \begin{center}
            $ \rho(\eta')= 2z-2$, \quad $\rho(w')=z-1$, \quad 
            $ \rho(\eta''')=2z$, \quad $\rho(w''')=z+1 $, \quad  
            $ \rho((\eta w)') = 3z-2$, \quad and \quad $\rho(ww')=2z-1 $.
        \end{center}
        Now, $\rho(ww')>\rho(\eta')$, so it is similarly deduced that  
        $\tilde{h}_{2, 2z-1} = z k_z^2$ which implies $k_{z} = 0$. 
        Therefore,
        \begin{center}
            $ \rho(\eta')= 2z-2$, \quad $\rho(w')=z-2$, \quad 
            $ \rho(\eta''')=2z$, \quad $\rho(w''')=z $, \quad  
            $ \rho((\eta w)') = 3z-3$, \quad and \quad $\rho(ww')=2z-3 $.
        \end{center}
        Again $\rho(\eta')>\rho(ww')$ so $j_{2z-1} = 0$. 
        Consequently,
        \begin{center}
            $ \rho(\eta')= 2z-3$, \quad $\rho(w')=z-2$, \quad 
            $ \rho(\eta''')=2z-1$, \quad $\rho(w''')=z $, \quad  
            $ \rho((\eta w)') = 3z-4$, \quad and \quad $\rho(ww')=2z-3 $.
        \end{center}
    \end{itemize}
    So, regardless of whether 
    $j_{2z} = 0$ or $k_{z} = 0$, the same order reduction is achieved. 
    After $i$ iterations of this argument, the orders become
    \begin{center}
        $ \rho(\eta')= 2z-3-2i$, \quad $\rho(w')=z-2-i$, \quad 
        $ \rho(\eta''')=2z-1-2i$, \quad $\rho(w''')=z-i $, \quad  
        $ \rho((\eta w)') = 3z-4-3i$, \quad and \quad $\rho(ww')=2z-3-2i $.
    \end{center}
    By balancing the orders, there are two instances 
    in which this argument breaks. From the first equation, the argument
    breaks when $\rho((\eta w)') \leq \rho(\eta''')$, and
    from the second equation, the argument 
    will break when $\rho(ww') = \rho(\eta') \leq \rho(w''')$.
    \begin{center}
        $\rho((\eta w)') \leq \rho(\eta'''):
            \quad 3z-4-3i \leq 2z-1-2i \quad \iff \quad  z-3 \leq i$; \\
        $\rho(ww') = \rho(\eta') \leq \rho(w'''):
            \quad 2z-3-2i \leq z-i \quad \iff \quad z-3 \leq i$. \\
    \end{center}
    So this argument will break after $(z-3)$ iterations for both cases.
    Plugging $i=z-3$ and recalculating the orders a final time yields
    \begin{equation*}
        \rho(\eta')= 3, \quad \rho(w')=1, \quad 
        \rho(\eta''')=5, \quad \rho(w''')=3 , \quad  
        \rho((\eta w)') = 5, \quad \text{and} \quad \rho(ww')=3.
    \end{equation*}
    Since $c \ne 0$ and $b^2 + d^2 \ne 0$, the 
    highest order terms in both equations are balanced. Furthermore,
    since $\rho(\eta') = 3$ and $\rho(w') = 1$, the identities in \eqref{derivatives}
    imply that $\rho(\eta) = 4$ and $\rho(w) = 2$.
    Therefore, if $c \ne 0$ and $b^2 + d^2 \ne 0$,
    then $j_r = 0$ for all $r \geq 5$ and $k_r =0$ for all $r \geq 3$. \\[10 pt]
    \textit{(ii)}
    If $b^2 + d^2 = 0$, then $b = d = 0$. Recall that the parameters $a,\, b,\, c,$ 
    and $d$ must satisfy 
    the three conditions given in \eqref{abcd-parameters}. Since $b=c=d=0$, it follows 
    that $a=\frac{1}{3}$. With this, system \eqref{eta-w-system-c0} reduces to
    \begin{equation*}\begin{cases}
        -\sigma \eta' + w' + (\eta w)' + \dfrac{1}{3} w''' = 0 \,,\\
        -\sigma w' + \eta' + ww' = 0 \,.
    \end{cases}\end{equation*}
    Now, $\tilde{h}_{1,5}$ and $\tilde{h}_{2,3}$ are given by
    \begin{equation*}
        \tilde{h}_{1,5} = 6 j_4 k_2 \quad \text{and} \quad 
        \tilde{h}_{2,3} = 4 j_4 + 2 k_2^2.
    \end{equation*}
    Requiring these both to be zero implies that $j_4 = k_2 = 0$. Substituting this 
    into the rest of the coefficients, it follows that that $\tilde{h}_{2,2} = 3 j_3$ which 
    implies $j_3 = 0$. Now, $\tilde{h}_{1,2}$ and $\tilde{h}_{2,1}$ are given by
    \begin{equation*}
        \tilde{h}_{1,2} = -k_1 (2\lambda^2 m^2 - 3j_2) \quad \text{and} \quad
        \tilde{h}_{2,1} = {k_1}^2 + 2 j_2.
    \end{equation*}
    Requiring $\tilde{h}_{2,1}=0$ yields $j_2 = -\frac{1}{2} {k_1}^2$, and substituting this into
    $\tilde{h}_{1,2}$ gives 
    \begin{equation*}
        \tilde{h}_{1,2} = -\frac{1}{2} k_1 (4 \lambda^2 m^2 + 3 {k_1}^2).
    \end{equation*} 
    Since $\lambda, m >0$ and $k_1 \in \R$, it follows that $4 \lambda^2 m^2 + 3 {k_1}^2 
    > 0$. Thus, $k_1 = 0$, and similarly, $j_2 = 0$. 
    Finally $\tilde h_{2,0} = j_1$, forcing $j_1 = 0$.
    
    Therefore, when $c \ne 0$ and $b^2 + d^2 = 0$, it is deduced that
    $j_r = k_r = 0$ for all integers $r \geq 1$ which implies both $\eta$ and $w$
    are constant and consequently, only the trivial solution exists. \\
\end{proof}

The above theorem shows that periodic traveling-wave solutions to 
\eqref{eta-w-system-c0} may only occur when $c = 0$ and $b^2 + d^2 \ne 0$, so 
we will focus on this case. Substituting $j_r = 0$ for all $r \geq 5$ and $k_r = 0$
for all $r \geq 3$ into \eqref{eqn:h-pq} yields
\begin{equation*}
    \left\{
    \begin{matrix}
    \begin{split}
        &- \lambda \sna \dna \sum_{q = 0}^{5} \tilde{h}_{1, q} \cnap{q} = 0, \\
        &- \lambda \sna \dna \sum_{q = 0}^{3} \tilde{h}_{2, q} \cnap{q} = 0.
    \end{split}
    \end{matrix}
    \right.
\end{equation*}

Next, by demanding $\tilde{h}_{1,q}=0$ and $\tilde{h}_{2,q}=0$, 
we obtain the following system of 10 equations 
and 11 unknowns, $\lambda,\, m,\, \sigma,\, j_i$ for $i=0,1,\dots,4$, and 
$k_i$ for $i=0,1,2$:
\begin{equation}\label{coeffs 2}\begin{cases}
    \tilde{h}_{1, 5} = -6\,j_{{4}} \left( 20\,b{\lambda}^{2}{m}^{2}\sigma-k_{{2}} \right) 
    \,, \\[5 pt]

    \tilde{h}_{1, 4} = -60\,b{\lambda}^{2}{m}^{2}\sigma\,j_{{3}}+5\,j_{{3}}k_{{2}}+5\,j_{{4}}
    k_{{1}}
    \,, \\[5 pt]

    \tilde{h}_{1, 3} = -24\,b{\lambda}^{2}{m}^{2}\sigma\,j_{{2}}+128\,b{\lambda}^{2}{m}^{2}
    \sigma\,j_{{4}}-24\,a{\lambda}^{2}{m}^{2}k_{{2}}-64\,b{\lambda}^{2}
    \sigma\,j_{{4}}-4\,\sigma\,j_{{4}}+4\,j_{{2}}k_{{2}}\\
    \qquad +4\,j_{{3}}k_{{1}}+4\,j_{{4}}k_{{0}}
    \,, \\[5 pt]

    \tilde{h}_{1, 2} = -6\,b{\lambda}^{2}{m}^{2}\sigma\,j_{{1}}+54\,b{\lambda}^{2}{m}^{2}
    \sigma\,j_{{3}}-6\,a{\lambda}^{2}{m}^{2}k_{{1}}-27\,b{\lambda}^{2}
    \sigma\,j_{{3}}-3\,\sigma\,j_{{3}}+3\,j_{{1}}k_{{2}}\\
    \qquad +3\,j_{{2}}k_{{1}}+3\,j_{{3}}k_{{0}}
    \,, \\[5 pt]

    \tilde{h}_{1, 1} = 16\,b{\lambda}^{2}{m}^{2}\sigma\,j_{{2}}-24\,b{\lambda}^{2}{m}^{2}
    \sigma\,j_{{4}}+16\,a{\lambda}^{2}{m}^{2}k_{{2}}-8\,b{\lambda}^{2}
    \sigma\,j_{{2}}+24\,b{\lambda}^{2}\sigma\,j_{{4}}\\
    \qquad -8\,a{\lambda}^{2}k_{
    {2}}-2\,\sigma\,j_{{2}}+2\,j_{{0}}k_{{2}}+2\,j_{{1}}k_{{1}}+2\,j_{{2}}
    k_{{0}}+2\,k_{{2}}
    \,, \\[5 pt]

    \tilde{h}_{1, 0} = 2\,b{\lambda}^{2}{m}^{2}\sigma\,j_{{1}}-6\,b{\lambda}^{2}{m}^{2}\sigma
    \,j_{{3}}+2\,a{\lambda}^{2}{m}^{2}k_{{1}}-b{\lambda}^{2}\sigma\,j_{{1}
    }+6\,b{\lambda}^{2}\sigma\,j_{{3}}-a{\lambda}^{2}k_{{1}}-\sigma\,j_{{1
    }}\\
    \qquad +j_{{0}}k_{{1}}+j_{{1}}k_{{0}}+\sigma k_{{1}}^{2}\,k_{{2}}
    \,, \\[5 pt]

    \tilde{h}_{2, 3} = -24\,d{\lambda}^{2}{m}^{2}\sigma\,k_{{2}}+2\,k_{{2}}^{2}+4\,j_{{4}}
    \,, \\[5 pt]

    \tilde{h}_{2, 2} = -6\,d{\lambda}^{2}{m}^{2}\sigma\,k_{{1}}+3\,k_{{1}}k_{{2}}+3\,j_{{3}}
    \,, \\[5 pt]

    \tilde{h}_{2, 1} = 16\,d{\lambda}^{2}{m}^{2}\sigma\,k_{{2}}-8\,d{\lambda}^{2}\sigma\,k_{{
        2}}-2\,\sigma\,k_{{2}}+2\,k_{{0}}k_{{2}}+k_{{1}}^{2}+2\,j_{{2}}
        \,, \\[5 pt]

    \tilde{h}_{2, 0} = 2\,d{\lambda}^{2}{m}^{2}\sigma\,k_{{1}}-d{\lambda}^{2}\sigma\,k_{{1}}-
    \sigma\,k_{{1}}+k_{{0}}k_{{1}}+j_{{1}}
    \,.
\end{cases}\end{equation}

Numerical computation shows that if $j_1 \ne 0$ or $j_3 \ne 0$, then there does not 
exist a solution to \eqref{coeffs 2}. Similarly, numerical computation shows if $k_1 \ne 0$, 
then the only solution to \eqref{coeffs 2} is one such that $\sigma = 0$, a contradiction.
So, we will proceed under the assumption that $j_1 = j_3 = k_1 = 0$. In particular, the only 
non-trivial periodic traveling wave solutions to \eqref{eta-w-system-c0} are of the form 
\begin{equation}\label{eqn:eta-w-2}
    \eta(\x) = j_0 + j_2 \cn^2 (\l \x, m) + j_4 \cn^4 (\l \x, m) \\ 
     \quad \text{and} \quad
    w(\x) = k_0 + k_2 \cn^2 (\l \x, m)\,.
\end{equation}
Moreover, this reduces \eqref{coeffs 2} to the following system of 5 equations with 8 unknowns:
\begin{equation}\label{coeffs 2 red}\begin{cases}
    \bar{h}_{1, 5} = -6\,j_{{4}} \left( 20\,b{\lambda}^{2}{m}^{2}\sigma-k_{{2}} \right) 
    \,, \\[5 pt]

    \bar{h}_{1, 3} = -24\,b{\lambda}^{2}{m}^{2}\sigma\,j_{{2}}+128\,b{\lambda}^{2}{m}^{2}
    \sigma\,j_{{4}}-24\,a{\lambda}^{2}{m}^{2}k_{{2}}-64\,b{\lambda}^{2}
    \sigma\,j_{{4}}-4\,\sigma\,j_{{4}}+4\,j_{{2}}k_{{2}}\\
    \qquad +4\,j_{{4}}k_{{0}}
    \,, \\[5 pt]

    \bar{h}_{1, 1} = 16\,b{\lambda}^{2}{m}^{2}\sigma\,j_{{2}}-24\,b{\lambda}^{2}{m}^{2}
    \sigma\,j_{{4}}+16\,a{\lambda}^{2}{m}^{2}k_{{2}}-8\,b{\lambda}^{2}
    \sigma\,j_{{2}}+24\,b{\lambda}^{2}\sigma\,j_{{4}}\\
    \qquad -8\,a{\lambda}^{2}k_{
    {2}}-2\,\sigma\,j_{{2}}+2\,j_{{0}}k_{{2}}+2\,j_{{2}}
    k_{{0}}+2\,k_{{2}}
    \,, \\[5 pt]

    \bar{h}_{2, 3} = -24\,d{\lambda}^{2}{m}^{2}\sigma\,k_{{2}}+2\,{k_{{2}}}^{2}+4\,j_{{4}}
    \,, \\[5 pt]

    \bar{h}_{2, 1} = 16\,d{\lambda}^{2}{m}^{2}\sigma\,k_{{2}}-8\,d{\lambda}^{2}\sigma\,k_{{
        2}}-2\,\sigma\,k_{{2}}+2\,k_{{0}}k_{{2}}+2\,j_{{2}}
        \,, \\[5 pt]

    \bar{h}_{1,4} = \bar{h}_{1,2} = \bar{h}_{1,0} = \bar{h}_{2,2} = \bar{h}_{2,0} = 0\,,
\end{cases}\end{equation}
where $\bar{h}_{1, q} = \tilde{h}_{1,q} \rvert_{j_1 = j_3 = k_1 = 0}$ and
$\bar{h}_{2, q} = \tilde{h}_{2,q} \rvert_{j_1 = j_3 = k_1 = 0}$.

The exact periodic traveling-wave solutions to \eqref{eta-w-system-c0} could 
then be established by solving the system of nonlinear equations \eqref{coeffs 2 red} 
with the help of the computer software Maple. As there are three degrees of freedom,
any combination of the 8 unknowns could be chosen as free parameters. For consistency,
we will choose $\lambda$, $m$, and $\sigma$ as free parameters, like in the previous section.

\section{Exact Jacobi Elliptic Solutions}
Finally, explicit solutions to the $abcd$-system \eqref{abcd-system} are given for the cases 
$c \ne 0$ and $c=0$.
Discussion of these solutions, and how they limit to solitary wave solutions, will be left 
for the next section. In addition to the explicit solutions being stated, graphs of these solutions 
will also be provided.

\subsection{The case when $c \ne 0$}
Solving \eqref{coeffs 1}, our solution will take the form as stated by 
\eqref{eqn:eta-w-1}, where the parameters are found to be given by one of the 
following sets.
\begin{enumerate}
        \item When $ac(b-6d)(3b-2d),\, c(2m^2 - 1)(b+2d)(3b+2d) < 0$
            and $ (2m^2 - 1)(3b+2d)(b-d) \geq 0$
            with equality only when $b=d$\,:
        \begin{equation}\label{soln:4.1.1}
            \begin{cases}
                j_{{0}}=-{\dfrac {a\,(3\,b+2\,d)\,(21\,b-46\,d)+ 2\,c\,(3\,b-2\,d)\,(b-6\,d)}
                    {2 c \left( 3\,b-2\,d \right)  \left( b-6\,d \right)}}\,, \\[9 pt]

                j_{{1}}=- \tau_1 \tau_2 {\dfrac {12\,a\, m \left( 3\,b+2\,d \right)\sqrt {
                    \left( 2\,{m}^{2}-
                    1 \right)  \left( 3
                    \,b+2\,d \right)  \left( b-d \right) } }{ c
                    \left( b-6\,d \right) \left( 3\,b-2\,d \right)
                    \left( 2\,{m}^{2}-1 \right) }}
                    \,, \\[9 pt]

                j_{{2}}= {\dfrac {9\,a\,{m}^{2} \left( 3\,b+2\,d \right)}{c \left( 3\,b-2
                    \,d \right)  \left( 2\,{m}^{2}-1 \right) }}
                    \,, \\[9 pt]

                k_{{0}}= \tau_{{1}}{\dfrac { \left( 21\,b+8\,c+14\,d \right) \sqrt 
                {-2\,a c \left( b-6\,d \right)  \left( 3\,b-2\,d \right)}}{ 2\,c  
                \left( b-6\,d \right) \left( 3\,b-2\,d \right)}}
                \,, \\[9 pt]

                k_{{1}}=\tau_2 {\dfrac {6\,m\sqrt {-2\,a c  
                    \left( b-6\,d \right)  \left( 3\,b-2\,d \right) \left( 2\,{m}^{2}-1 \right)
                    \left( 3\,b+2\,d
                    \right)  \left( b-d \right) }}{c \left( 2\,{m}^{2}-1 \right) 
                    \left( b-6\,d \right)  \left( 3\,b-2\,d \right) }}
                    \,, \\[9 pt]

                k_{{2}}= -\tau_1 {\dfrac {9 \,m^2 \left( 3\,b+2\,d \right) \sqrt {-2\,a c 
                    \left( b-6\,d \right) \left( 3\,b-2\,d \right)}}{c \left( 2\,{m}^{2}-1 
                    \right)  \left( b-6\,d \right)  \left( 3\,b-2\,d\right) }}
                    \,, \\[9 pt]
                
                \lambda=\dfrac{1}{2}\sqrt {{\dfrac {-6\,\left( 3\,b+2\,d \right)}
                    {c \left( 2\,{m}^{2}-1\right)  \left( b+2\,d \right) }}}
                \,, \\[9 pt] 

                \sigma= \tau_1 {\dfrac {4\sqrt {-2\,a c \left( b-6\,d \right)  \left( 
                    3\,b-2\,d \right)}}{ \left( b-6\,d \right)  \left( 3\,b-2\,d
                     \right) }}
                    \,, \\[9 pt]
                    
                m \in (\,0,1\,] \,.
            \end{cases}
        \end{equation}
        Finding this solution requires taking the square root of two terms, which leads to 
        two independent plus or minuses. To avoid confusion, we define 
        \begin{equation*}
            \tau_1 = \{ \, 1,\, -1 \, \} \quad \text{and} \quad \tau_2 = \{ \, 1,\, -1 \, \},
        \end{equation*}
        to denote these independent plus or minuses.
        Two graphs of this solution can be found below in Figures 
        \ref{4.1.1a} and \ref{4.1.1b}, in which $\eta$ is graphed in blue and $w$ 
        in green. The explicit parameters chosen for each graph are given in the captions. 
        \begin{figure}[h]
            \centering
            \begin{subfigure}[b]{0.475\textwidth}
                \includegraphics[width=\textwidth]{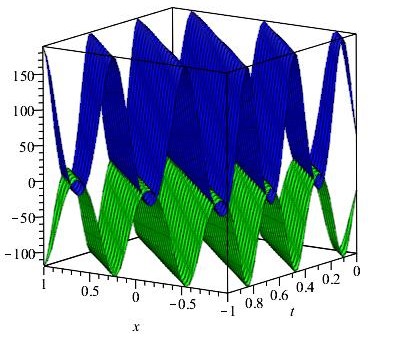}
                \caption{$\{ m=\frac{3}{4},\, a=-\frac{5}{6},\, b=1,\, c=-\frac{5}{6},\, d=1,$
                    \\ $\tau_1 = 1,\, \tau_2 = 1\}$}\label{4.1.1a}
            \end{subfigure}
            \begin{subfigure}[b]{0.475\textwidth}
                \includegraphics[width=\textwidth]{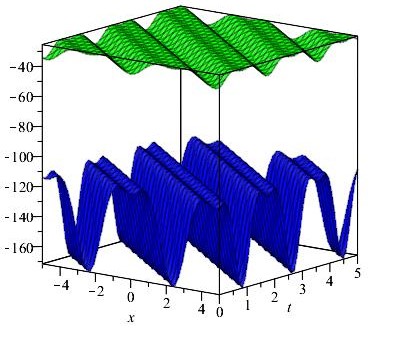}
                \caption{$\{ m=\frac{1}{4},\, a=-7,\, b=2,\, c=\frac{4}{3},\, d=4,$ 
                    \\ $\tau_1 = -1,\, \tau_2 = 1\}$}\label{4.1.1b}
            \end{subfigure}
            \caption{Graphs of solution \eqref{soln:4.1.1}}
        \end{figure} 

    \item When $8 a c + \sigma^2(b-2d) ^{2} > 0\,$: 
        \begin{equation}\label{soln:4.1.2}
            \begin{cases}
                j_1 = k_1 = 0\,, \\[9 pt]

                j_{{0}}= {\dfrac {1}{ 4\,c^2\left( 4\,a\,c + \sigma^2 \left({b}^{2}
                +4\,{d}^{2} \right) \pm
                \sigma \left( b+2\,d \right) \sqrt {8\,a\,c +  \sigma^2 
                \left( b-2\,d \right) ^{2}} \right)}} 
                \Big(-8\,{b}^{4}c\,{\lambda}^{2}{m}^{2}{\sigma}^{4} \\[9 pt]

                \qquad +16\,{b}^{3}c\,d\,{\lambda}^{2
                }{m}^{2}{\sigma}^{4}-64\,a\,{b}^{2}{c}^{2}{\lambda}^{2}{m}^{2}{\sigma}^{
                2}-32\,a\,b\,{c}^{2}d\,{\lambda}^{2}{m}^{2}{\sigma}^{2}-64\,a\,{c}^{2}{d}^{2}{
                \lambda}^{2}{m}^{2}{\sigma}^{2}\\[5 pt]

                \qquad +4\,{b}^{4}c\,{\lambda}^{2}{\sigma}^{4}-8
                \,{b}^{3}c\,d\,{\lambda}^{2}{\sigma}^{4}-64\,{a}^{2}{c}^{3}{\lambda}^{2}{m
                }^{2}+32\,a\,{b}^{2}{c}^{2}{\lambda}^{2}{\sigma}^{2}+16\,a\,b\,{c}^{2}d\,{
                \lambda}^{2}{\sigma}^{2}\\[5 pt]
                
                \qquad +32\,a\,{c}^{2}{d}^{2}{\lambda}^{2}{\sigma}^{2}+
                {b}^{4}{\sigma}^{4}-4\,{b}^{3}d\,{\sigma}^{4}+4\,{b}^{2}{d}^{2}{\sigma}^
                {4}+32\,{a}^{2}{c}^{3}{\lambda}^{2}+8\,a\,{b}^{2}c\,{\sigma}^{2}-8\,a\,b\,c\,d\,{
                \sigma}^{2}\\[5 pt]
                
                \qquad -4\,{b}^{2}{c}^{2}{\sigma}^{2}-16\,{c}^{2}{d}^{2}{\sigma}^{
                2}+8\,{a}^{2}{c}^{2}-16\,a\,{c}^{3} \mp \,\sigma \sqrt {8\,a\,c + \sigma^2 \left( b-2\,d \right) ^{2}}
                \Big( 8\,{b}^{3}c\,{\lambda}^{2}{m}^{2}{\sigma}^{2} \\[5 pt]

                \qquad +32\,a\,b\,{c}^{2}{\lambda}^{2}{m}^{2}+32\,a\,{c}^{2}d{\lambda}^{2}{m}^{2}-4\,{
                b}^{3}c\,{\lambda}^{2}{\sigma}^{2}-16\,a\,b\,{c}^{2}{\lambda}^{2}-16\,a\,{c}^{
                2}d\,{\lambda}^{2}-{b}^{3}{\sigma}^{2}\\[5 pt]

                \qquad +2\,{b}^{2}d\,{\sigma}^{2}-4\,a\,b\,c+4\,b\,{c}^{2}+8\,{c}^{2}d \Big) 
                \Big)\,, \\[9 pt]

                j_2 = \dfrac{3\, \lambda^2 m^2}{2\,c} \Big( 4\,a\,c + b\,{\sigma}^{2} 
                \left( b-2\,d \right) \pm b\,\sigma\,\sqrt {8\,a\,c + {\sigma}^{2}
                \left( b-2\,d \right) ^{2}} \Big)
                \,, \\[9 pt]

                k_{{0}}={\dfrac {1}{ 2\,c\left(\sigma \left( b+2\,d \right) 
                \pm \sqrt { 8\,a\,c + \sigma^2 \left( b-2\,d\right) ^{2}} 
                \right)}} \Big( -8\,{b}^{2}c\,{\lambda}^{2}{m}^{2}{\sigma}^{2}
                -32\,c\,{d}^{2}{\lambda}^{2}{m}^{2}{\sigma}^{2}\\[9 pt]
                
                \qquad -32\,a\,{c}^{2}{\lambda}^{2}{m}^{2}+4\,{b}^{2}c\,{
                \lambda}^{2}{\sigma}^{2}+16\,c\,{d}^{2}{\lambda}^{2}{\sigma}^{2}+16\,a\,{
                \lambda}^{2}{c}^{2}-{b}^{2}{\sigma}^{2}+2\,b\,c\,{\sigma}^{2}+2\,b\,d\,{\sigma
                }^{2} \\[5 pt]
                
                \qquad +4\,c\,d\,{\sigma}^{2}-4\,a\,c
                \mp \sigma \sqrt {8\,a\,c + {\sigma}^{2} \left( b-2\,d \right) ^{2}} 
                \Big( 8\,b\,c\,{\lambda}^{2}{m}^{2}+16\,c\,d\,{\lambda}^{2}{m}^{2}-4
                \,b\,c\,{\lambda}^{2} \\[5 pt]
                
                \qquad -8\,c\,d\,{\lambda}^{2}+b-2\,c \Big)\Big)
                \,, \\[9 pt]

                k_2 = 3\, \lambda^2 m^2 \left(\sigma\left( b+2\,d \right) \pm 
                \sqrt {8\,a\,c + \sigma^2 \left( b-2\,d \right) ^{2}}  \right)
                \,, \\[9 pt]
            
                \lambda > 0 ,\, \sigma \ne 0 ,\, m \in (\,0,1\,] \,.
            \end{cases}
        \end{equation}
        Above, the $\pm$ and $\mp$ should be chosen such that the top or bottom signs 
        are chosen together. 
        Two graphs of this solution can be found below in Figures 
        \ref{4.1.2a} and \ref{4.1.2b}, in which $\eta$ is graphed in blue and $w$ 
        in green. The explicit parameters chosen for each graph are given in the captions.
        \begin{figure}[h]
            \centering
            \begin{subfigure}[b]{0.475\textwidth}
                \includegraphics[width=\textwidth]{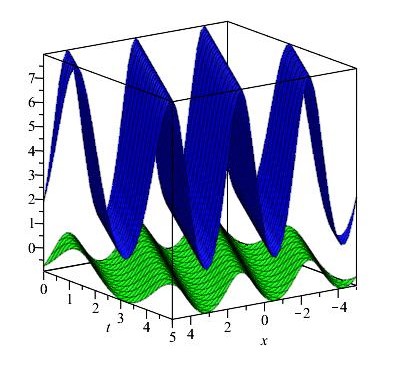}
                \caption{$\{ \lambda = 1,\, m=\frac{1}{\sqrt{2}},\, 
                    \sigma = 1,\, a=1,\, b=-\frac{8}{3},$\\ $c=1,\, d=1 \}$ and the 
                    top signs of $\pm$ and $\mp$}
                    \label{4.1.2a}
            \end{subfigure}
            \begin{subfigure}[b]{0.475\textwidth}
                \includegraphics[width=\textwidth]{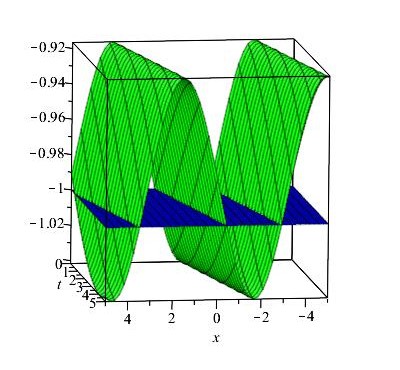}
                \caption{$\{ \lambda = \frac{1}{2},\, m=\frac{1}{4},\, 
                \sigma = -\frac{1}{3},\, a=0,\, b=-1,\, c=-\frac{2}{3},\, d=2\}$ and 
                the bottom signs of $\pm$ and $\mp$}
                \label{4.1.2b}
            \end{subfigure}
            \caption{Graphs of solution \eqref{soln:4.1.2}}
        \end{figure}
\end{enumerate}

\subsection{The case when $c=0$}
Solving \eqref{coeffs 2 red}, our solution will take the form as stated by 
\eqref{eqn:eta-w-2}, where the parameters are found to be given by one of the 
following sets.
\begin{enumerate}
    \item When $4 b - d \ne 0\,$:
    \begin{equation}\label{soln:4.2.1}
        \begin{cases}
            j_0 = \dfrac {1}{9\,{\sigma}^{2} \left( 4\,b-d \right) ^{2}}
            \Big(-32\, b\, \lambda^4 \sigma ^4 \left( 4\,b-d\right) ^{2} \left( 5\,b-3\,d \right)
            \left( 11\,{m}^{4}-11\,{m}^{2}-4 \right) \\[9 pt]
            \qquad + 3\,\sigma^2 \left( 4\,b-d \right)
            \big[  3\,d - 4\,b\left( 3+5\,a\,\lambda^2 \left( 2\,{m}^{2}-1 \right)
            \right) \big] + 9\,a^2 \Big) \,, \\[9 pt]

            j_2 = {\dfrac {20\,b\,{\lambda}^{2}{m}^{2}\left(  3\,a + 4\, \lambda^2 
                \sigma^2   \left( 4\,b-d \right) \left( 5\,b-3\,d \right) 
                \left( 2\,{m}^{2}-1 \right) \right)}{3 (4\,b-d)}  }
               \,, \\[9 pt]

            j_4 = -40\,b\,{\lambda}^{4}{m}^{4}{\sigma}^{2} \left( 5\,b-3\,d
            \right) \,, \\[9 pt]

            k_{{0}}= {\dfrac { -3\,a + \sigma^2 \left( 4\,b-d \right) 
            \left( {3} - 20\,b\,{\lambda}^{2} 
            \left( 2{m}^{2}-1 \right) \right)   
            }{ 3 \,\sigma \left( 4\,b-d \right)} }\,, \\[9 pt]

            k_{{2}}= 20\,b\,{\lambda}^{2}{m}^{2}\sigma \,, \\[9 pt]                    
            
            \lambda > 0 ,\, \sigma \ne 0 ,\, m \in (\,0,1\,] \,.
        \end{cases}
    \end{equation}
    Two graphs of this solution can be found below in Figures 
    \ref{4.2.1a} and \ref{4.2.1b}, in which $\eta$ is graphed in blue and $w$ 
    in green. The explicit parameters chosen for each graph are given in the captions. 
    \begin{figure}[h]
        \centering
        \begin{subfigure}[b]{0.475\textwidth}
            \includegraphics[width=\textwidth]{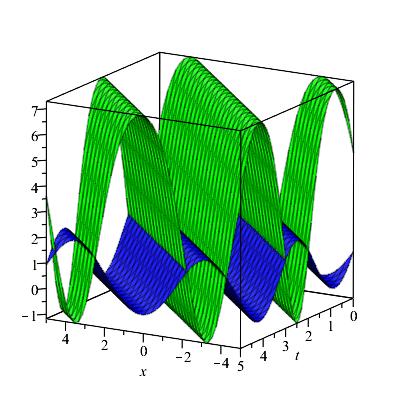}
            \caption{$\{ \lambda = \frac{1}{2},\, m=\frac{1}{2},\, 
            \sigma = -2,\, a=1, b=-1,\, d=\frac{1}{3}\}$}\label{4.2.1a}
        \end{subfigure}
        \begin{subfigure}[b]{0.475\textwidth}
            \includegraphics[width=\textwidth]{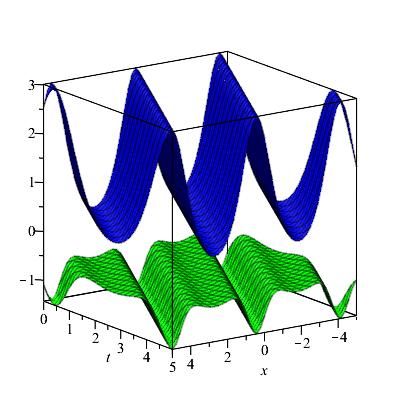}
            \caption{$\{ \lambda = 1,\, m=\frac{9}{10},\, 
            \sigma = 1,\, a=0, b=\frac{1}{6},\, d=\frac{1}{6}\}$}\label{4.2.1b}
        \end{subfigure}
        \caption{Graphs of solution \eqref{soln:4.2.1}}
    \end{figure}

    \item When $b - 2 d \ne 0\,$:
        \begin{equation}\label{soln:4.2.2}
            \begin{cases}
                j_{{0}}= {\dfrac {a^2 - \sigma^2 \left( b-2\,d \right)  \left(b - 2\, d 
                \left( 1 + 2\,a\,\lambda^2 \left(2\, {m}^{2}-1 \right)\right)  \right)
                }{ \sigma^2 \left( b-2\,d \right) ^{2}}}
                \,, \\[9 pt]

                j_2 = -12\,{\dfrac {a\,d\,{\lambda}^{2}{m}^{2}}{b-2\,d}}\,, \\[9 pt]

                k_{{0}}= {\dfrac {a + \sigma^2 \left( b-2\,d \right)  \left( 1 - 4
                \, d\,{\lambda}^{2}
                \left( 2\,{m}^{2}-1 \right) \right) }{\sigma \left( 
                b-2\,d \right) }}
                \,, \\[9 pt] 

                k_2 = 12\,d\,{\lambda}^{2}{m}^{2}\sigma\,, \\[9 pt]

                \lambda > 0 ,\, \sigma \ne 0 ,\, m \in (\,0,1\,] \,.
            \end{cases}
        \end{equation}
        Two graphs of this solution can be found below in Figures 
        \ref{4.2.2a} and \ref{4.2.2b}, in which $\eta$ is graphed in blue and $w$ 
        in green. The explicit parameters chosen for each graph are given in the captions. 
        \begin{figure}[h]
            \centering
            \begin{subfigure}[b]{0.475\textwidth}
                \includegraphics[width=\textwidth]{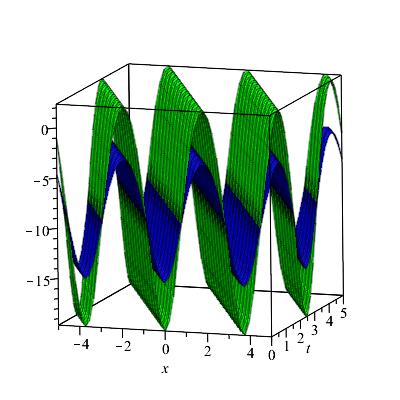}
                \caption{$\{ \lambda = 1,\, m=\frac{1}{\sqrt{2}},\, 
                \sigma = -1, a=-\frac{11}{3},\, b=2,$ \\ $d=2\}$}\label{4.2.2a}
            \end{subfigure}
            \begin{subfigure}[b]{0.475\textwidth}
                \includegraphics[width=\textwidth]{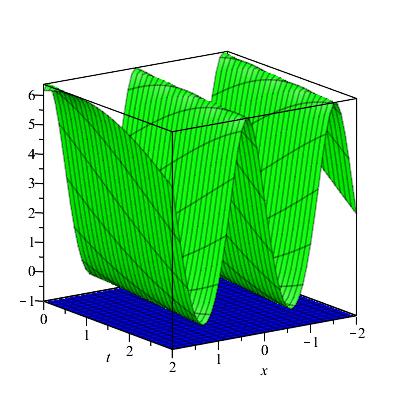}
                \caption{$\{ \lambda = 2,\, m=\frac{3}{4},\, 
                \sigma = \frac{1}{8},\, a=0, \, b=-\frac{5}{3},$ 
                \\ $d=2\}$}\label{4.2.2b}
            \end{subfigure}
            \caption{Graphs of solution \eqref{soln:4.2.2}}
        \end{figure}
\end{enumerate}

\subsection{Semi-trivial Solutions}
Substituting $a = b = 0$ into \eqref{coeffs 1} and solving, a solution 
of the form \eqref{eqn:eta-w-ST} is found where
\begin{equation}\label{soln:4.3}\begin{cases}
    j_0 = -1\,, \\[5 pt]

    k_{{0}}=-8\,d\,{\lambda}^{2}{m}^{2}\sigma+4\,d\,{\lambda}^{2}\sigma
        +\sigma \,, \\[5 pt]

    k_1 = 0 \,, \\[5 pt]

    k_2 = 12\,d\,{\lambda}^{2}{m}^{2}\sigma\,, \\[5 pt]

    \lambda > 0 ,\, \sigma \ne 0 ,\, m \in (\,0,1\,] \,.
\end{cases}\end{equation}
One may notice that Figures \ref{4.1.2b} and \ref{4.2.2b} are graphs in which 
$\eta = -1$. In fact, these are exactly graphs of solution \eqref{soln:4.3}, 
so additional graphs need not be provided here.
The relationship between solution \eqref{soln:4.3} and 
solutions \eqref{soln:4.1.2} and \eqref{soln:4.2.2} is made clear in 
the next section.

\section{Conclusion}
As stated at the end of Section 4.3, there is a distinct relationship between 
solutions \eqref{soln:4.1.2}, \eqref{soln:4.2.2}, and \eqref{soln:4.3}.
In fact, if $\sigma (b-2d) > 0$, then solution \eqref{soln:4.2.2} can be retrieved 
by taking the limit as $c$ approaches zero of solution \eqref{soln:4.1.2} while choosing the 
bottom signs for $\pm$ and $\mp$.  Similarly, solution \eqref{soln:4.3} can be recovered 
from solutions \eqref{soln:4.1.2} or \eqref{soln:4.2.2} in any of the following ways
\begin{itemize}
    \item if $d \sigma>0$, then evaluating solution \eqref{soln:4.1.2} at $a=b=0$ and 
        choosing the top signs for $\pm$ and $\mp$;
    \item if $\sigma (b-2d) > 0$, then evaluating solution \eqref{soln:4.1.2} at 
        $a=0$ and choosing the bottom signs for $\pm$ and $\mp$;
    \item if $b - 2d \ne 0$, then evaluating solution \eqref{soln:4.2.2} at 
        $a=0$.
\end{itemize}

The semi-trivial solution \eqref{soln:4.3} is also of importance.
When we evaluate system \eqref{abcd-system} at $\eta = -1$ and $a = 0$, it 
collapses to a single equation. This equation is 
\begin{equation*}
    w_t + w w_x - d w_{xxt} = 0\,,
\end{equation*}
which is exactly the BBM equation after a transformation. This equation has been shown 
to have cnoidal wave solutions \cite{AZ}, so it is especially important 
that we recovered a cnoidal solution for this case.

Another fundamental property of the Jacobi cnoidal function is that as $m$ approaches
one, the cnoidal function limits to the hyperbolic secant function. This is important 
because solitary-wave solutions are often given in terms of hyperbolic secant functions.
Thus, in principle, taking the limit of our solutions as $m$ approaches one may
retrieve solitary-wave solutions. Indeed, for \eqref{soln:4.1.2} and \eqref{soln:4.2.2} we get 
solitary-wave solutions of the form 
\begin{equation}\label{soli-wave-form-1}
    \eta = \bar{j}_0 + \bar{j}_2 \sech^2(\bar{\lambda} (x - \bar{\sigma} t))
    \quad \text{and} \quad w = \bar{k}_0 + \bar{k}_2 \sech^2(\bar{\lambda} (x - 
    \bar{\sigma} t),
\end{equation}
where the bar indicates the evaluation at $m=1$. This straightforward 
computation is left to the reader. Some solitary-wave solutions for these 
solutions are graphed in Figures \ref{soli_4.1.2} and \ref{soli_4.2.2}, 
in which $\eta$ is graphed in blue and $w$ in green. The exact values chosen 
for the parameters are given in the captions of each figure. 
In the paper \cite{C1, C2}, solitary-wave solutions were shown to be given exactly 
by the same form in \eqref{soli-wave-form-1} for cases $\{a=0,\, c\ne0 \}$, $\{b=d=0,\, 
a=c\}$ and $\{b=c=0,\, d\ne0\}$ where the coefficients solve certain equations.
\begin{figure}[h]
    \centering
    \begin{subfigure}[b]{0.475\textwidth}
        \includegraphics[width=\textwidth]{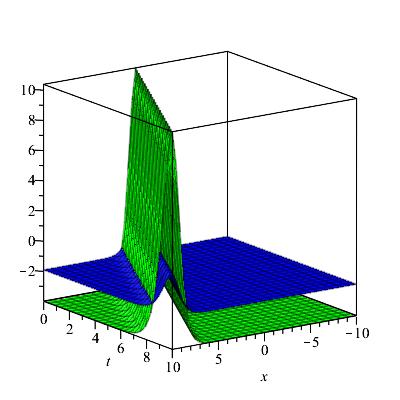}
        \caption{Solution \eqref{soln:4.1.2} when $\{ \lambda = 1,\, m=1,\, 
        \sigma = 1,\\ a=1,\, b=-\frac{8}{3},\, c=1,\, d=1 \}$ and 
        the top signs \\ of $\pm$ and $\mp$ are chosen}
        \label{soli_4.1.2}
    \end{subfigure}
    \begin{subfigure}[b]{0.475\textwidth}
        \includegraphics[width=\textwidth]{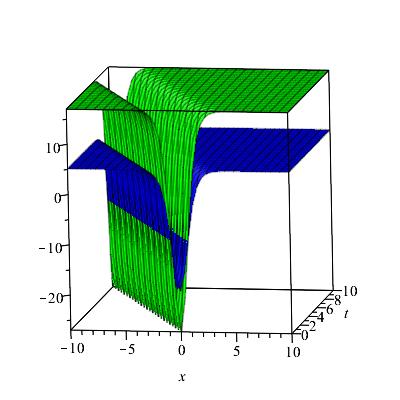}
        \caption{Solution \eqref{soln:4.2.2} when $\{ \lambda = 1,\, m=1,\, 
        \sigma = -1,\\ a=-\frac{11}{3},\, b=2,\, d=2\}$}\label{soli_4.2.2}
    \end{subfigure}
    \caption{Graphs of solitary wave solutions of form \eqref{soli-wave-form-1}}
\end{figure}

From solution \eqref{soln:4.2.1}, we get a solitary-wave solution of the form
\begin{equation}\label{soli-wave-form-2}
    \eta = \bar{j}_0 + \bar{j}_2 \sech^2(\lambda (x - \sigma t)) + \bar{j}_4 \sech^4(\lambda (x - \sigma t))
    \quad \text{and} \quad w = \bar{k}_0 + \bar{k}_2 \sech^2(\lambda (x - \sigma t)),
\end{equation}
where again, the bar indicates the evaluation at $m=1$. This is graphed in Figure \ref{soli_4.2.1},
in which $\eta$ is graphed in blue and $w$ in green.
The form found in this paper \eqref{soli-wave-form-2} again matches the form of solitary wave 
solutions found in \cite{C1, C2} for cases $\{a=c=0\}$ and $\{c=d=0,\, b\ne0\}$. This also
re-emphasizes exactly what was found in this paper, that solutions with a $\cn^4$
term in $\eta$ can only occur when $c=0$.

Similarly, looking at solution \eqref{soln:4.3} as $m$ approaches one, we get a 
solitary-wave solution of the form
\begin{equation*}
    \eta = -1
    \quad \text{and} \quad w = \bar{k}_0 + \bar{k}_2 \sech^2(\lambda (x - \sigma t)),
\end{equation*}
which was also found in \cite{C1, C2} for the case $\{a=0\}$.

Of course, there is no one-to-one correspondence between cnoidal solutions 
and solitary-wave solutions. Thus, one should not expect to establish \textit{all} 
solitary-wave solutions by taking the limit as the elliptic modulus $m$ approaches one.
One such instance of this is found in \eqref{soln:4.1.1}, where a solution of the 
form 
\begin{equation}\label{soli-wave-form-3}
    \eta = \bar{j}_0 + \bar{j}_1 \sech (\bar{\lambda} (x - \bar{\sigma} t)) 
        + \bar{j}_2 \sech^2(\bar{\lambda} (x - \bar{\sigma} t))
    \quad \text{and} \quad w = \bar{k}_0 + \bar{k}_1 \sech (\bar{\lambda} 
        (x - \bar{\sigma} t)) + \bar{k}_2 \sech^2(\bar{\lambda} (x - \bar{\sigma} t))\,,
\end{equation} 
is achieved as $m$ approaches one. This is graphed in Figure \ref{soli_4.1.1},
in which $\eta$ is graphed in blue and $w$ in green. 
This form of solution was 
not found in \cite{C1, C2}, showing the lack of correspondence.
In turn, a form of solitary-wave solutions was found in \cite{C1, C2} that was not 
recovered in this paper. This solution is found when $\{b=c=d=0,\, a>0\}$, and takes the form 
\begin{equation*}
    \eta = j_0 + j_2 \sech^2(\lambda (x - \sigma t)) \quad \text{and} \quad
    w = k_0 + k_1 \tanh(\lambda (x - \sigma t)).
\end{equation*}
This solution was not established in this paper because hyperbolic tangent is not
a limiting case of the Jacobi cnoidal function. 
\begin{figure}[h]
    \centering
    \begin{subfigure}[b]{0.475\textwidth}
        \includegraphics[width=\textwidth]{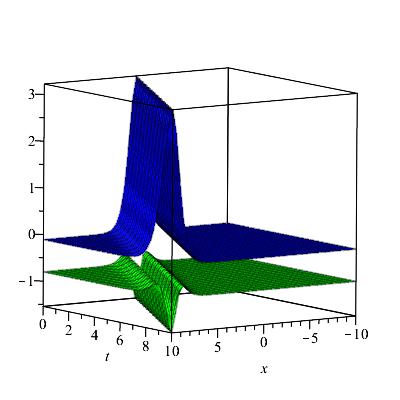}
        \caption{Solution \eqref{soln:4.2.1} when $\{ \lambda = 1,\, m=1,\, 
        \sigma = 1,\\ a=0,\, b=\frac{1}{6},\, d=\frac{1}{6}\}$}\label{soli_4.2.1}
    \end{subfigure}
    \begin{subfigure}[b]{0.475\textwidth}
        \includegraphics[width=\textwidth]{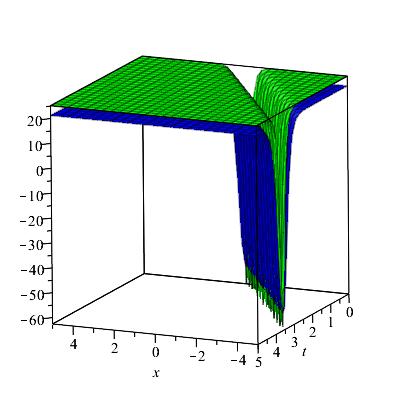}
        \caption{Solution \eqref{soln:4.1.1} when $\{m=1,\, a = -\frac{5}{6},
        \, b = 1,\\ c=-\frac{1}{6},\, d=\frac{1}{3},\, \tau_1 = 1,\, 
        \tau_2 = -1 \}$}\label{soli_4.1.1}
    \end{subfigure}
    \caption{Graphs of solitary wave solutions of form \eqref{soli-wave-form-2}
        and \eqref{soli-wave-form-3}}
\end{figure}

\noindent \textbf{\Large{Declarations}} \\

\noindent \textbf{Conflict of interest}: The authors have no relevant financial or non-financial interests to disclose. \\
\noindent \textbf{Data availability}:  Data sharing is not applicable to this article as no datasets were generated or analysed during this study.

\end{document}